\documentclass[oneside,british,reqno]{amsart}
\usepackage[latin9]{inputenc}
\usepackage{amsthm}
\usepackage{amstext}
\usepackage{amssymb}
\usepackage{esint}

\makeatletter
\numberwithin{equation}{section}
\numberwithin{figure}{section}
\theoremstyle{plain}
\newtheorem{thm}{\protect\theoremname}
  \theoremstyle{plain}
  \newtheorem{conjecture}[thm]{\protect\conjecturename}
  \theoremstyle{plain}
  \newtheorem{lem}[thm]{\protect\lemmaname}
  \theoremstyle{remark}
  \newtheorem{rem}[thm]{\protect\remarkname}
  \theoremstyle{remark}
  \newtheorem{claim}[thm]{\protect\claimname}
  \theoremstyle{plain}
  \newtheorem{prop}[thm]{\protect\propositionname}

 \author{Antonio Córdoba$^1$}
\thanks{1- The authors are partially supported by the grant
MTM2014-56350-P from the Ministerio de Ciencia e Innovaci\'{o}n (Spain).}
       
\address{Instituto de Ciencias Matemáticas CSIC-UAM-UC3M-UCM- Departamento de Matemáticas (Universidad Autónoma de Madrid), 28049 MAdrid, Spain}
       \email{antonio.cordoba@uam.es}

       \author{Eric Latorre$^1$}

       \address{Instituto de Ciencias Matemáticas CSIC-UAM-UC3M-UCM- Departamento de Matemáticas (Universidad Autónoma de Madrid), 28049 MAdrid, Spain}

       \email{eric.latorre@icmat.es}

\makeatother

\usepackage{babel}
  \providecommand{\claimname}{Claim}
  \providecommand{\conjecturename}{Conjecture}
  \providecommand{\lemmaname}{Lemma}
  \providecommand{\propositionname}{Proposition}
  \providecommand{\remarkname}{Remark}
\providecommand{\theoremname}{Theorem}

\begin{document}

\title[Radial Multipliers and Restriction in mixed norm spaces]{Radial Multipliers and restriction to surfaces of the Fourier transform
in mixed-norm spaces}
\begin{abstract}
In this article we revisit some classical conjectures in harmonic
analysis in the setting of mixed norm spaces $L_{rad}^{p}L_{ang}^{2}\left(\mathbb{R}^{n}\right)$.
We produce sharp bounds for the restriction of the Fourier transform
to compact hypersurfaces of revolution in the mixed norm setting and
study an extension of the disc multiplier. We also present some results
for the discrete restriction conjecture and state an intriguing open
problem.
\end{abstract}
\maketitle

\section{Introduction}

The well-known restriction conjecture, first proposed by E. M. Stein,
asserts that the restriction of the Fourier transform of a given integrable
function $f$ to the unit sphere, $\hat{f}|_{S^{n-1}}$, yields a
bounded operator from $L^{p}\left(\mathbb{R}^{n}\right)$, $n\geq2$,
to $L^{q}\left(S^{n-1}\right)$ so long as 
\[
\begin{tabular}{c}
 \end{tabular}1\leq p<\frac{2n}{n+1},\ \frac{1}{q}\geq\frac{n+1}{n-1}\left(1-\frac{1}{p}\right).
\]
This conjecture has been fully proved only in dimension $n=2$ by
C. Fefferman \cite{key-1} (see also \cite{key-2} for an alternative
geometrical proof). In higher dimensions, the best known result is
the particular case $q=2$ and $1\leq p\leq\frac{2\left(n+1\right)}{n+3}$,
which proof was obtained independently by P. Tomas and E. M. Stein
\cite{key-3}. 

The periodic analogue, i.e. for Fourier series, was observed by A.
Zygmund \cite{key-11},\textbf{ }but also in two dimensions. It asserts
that for any trigonometric polynomial 
\[
P\left(x\right)=\sum_{\left|\nu\right|=R}a_{\nu}e^{2\pi i\nu\cdot x},\ \nu\in\mathbb{Z}^{2},
\]
the following inequality holds:
\[
\left\Vert P\right\Vert _{L^{4}\left(Q\right)}\lesssim\left\Vert P\right\Vert _{L^{2}\left(Q\right)},
\]
uniformly on $R>0$ and where $Q$ is any unit square in the plane. 

The alternative proof given in \cite{key-2} allows us to connect
both the periodic and the nonperiodic restriction theorems, explaining
the reason for the apparently different numerologies of the corresponding
$\left(p,q\right)$ exponent ranges. It also raises an interesting
question about the location of lattice points in small arcs of circles
\cite{key-4}.

The first result in this paper goes further in that direction: Given
$\left\{ \xi_{j}\right\} $ a finite set of points in the circle $\left\{ \left\Vert \xi\right\Vert =R\right\} $
of the plane, let us consider 
\[
M:=\sup_{j}\#\left\{ \xi_{k},\ \left\Vert \xi_{k}-\xi_{j}\right\Vert \leq R^{\frac{1}{2}}\right\} .
\]
We have:
\begin{thm}
\label{thm:Discrete}The following inequality holds
\begin{equation}
\sup_{\mu\left(Q\right)=1}\left[\int_{Q}\left|\sum a_{k}e^{2\pi i\xi_{k}\cdot x}\right|^{4}d\mu\left(x\right)\right]^{\frac{1}{4}}\lesssim M^{\frac{1}{2}}\left(\sum\left|a_{k}\right|^{2}\right)^{\frac{1}{2}},\label{eq:Discrete}
\end{equation}
where the suprermum is taken over all unit squares of $\mathbb{R}^{2}$
and $\mu$ corresponds to the Lebesgue measure. 
\end{thm}
The corresponding result in higher dimensions ($n\geq3$) is an interesting
open problem: 
\begin{conjecture}
Let $\left\{ \xi_{j}\right\} \subset S_{R}^{n-1}$ and $M:=\sup_{j}\#\left\{ \xi_{k},\ \left\Vert \xi_{k}-\xi_{j}\right\Vert \leq R^{\frac{1}{2}}\right\} ,$
is it true that 
\begin{equation}
\sup_{\mu\left(Q\right)=1}\left[\int_{Q}\left|\sum a_{k}e^{2\pi i\xi_{k}\cdot x}\right|^{\frac{2n}{n-1}}d\mu\left(x\right)\right]^{\frac{n-1}{2n}}\lesssim M^{\frac{1}{2}}\left(\sum\left|a_{k}\right|^{2}\right)^{\frac{1}{2}}.
\end{equation}

\end{conjecture}
Although there are many interesting publications by several authors
throwing some light on the restriction conjecture, its proof remains
open in dimension $n\geq3$. One of the more remarkable improvements
was B. Barcelo's thesis \cite{key-5}. He proved that Fefferman's
result also holds for the cone in $\mathbb{R}^{3}$. Another interesting
result was given by L. Vega in his Ph.D. thesis \cite{key-12},\textbf{
}where he obtained the best result in the Stein-Tomas restriction
inequality when the space $L^{p}\left(\mathbb{R}^{n}\right)$ is replaced
by $L_{rad}^{p}L_{ang}^{2}\left(\mathbb{R}^{n}\right).$

Here we shall consider the restriction of the Fourier transform to
other surfaces of revolution in these mixed norm spaces. Several special
cases have already been treated \cite{key-6,key-7} but we present
a more general and unified proof for ``all'' compact surfaces of
revolution:
\[
\Gamma=\left\{ \left(g\left(z\right),\theta,z\right)\in\mathbb{R}^{n+1},\ \theta\in S^{n-1},\ a\leq z\leq b,\ 0\leq g\in C^{1}\left(a,b\right)\right\} .
\]
That is, in $\mathbb{R}^{n+1}$, $n\geq2$, we consider cylindrical
coordinates $\left(r,\theta,z\right)$ where the first components
$\left(r,\theta\right)$ correspond to the standard polar coordinates
in $\mathbb{R}^{n}$; $0<r<\infty,\ \theta\in S^{n-1}$, and $z\in\mathbb{R}$
denotes the zenithal coordinate. In this coordinate system, the $L_{rad}^{p}L_{zen}^{2}L_{ang}^{2}\left(\mathbb{R}^{n+1}\right)$
norm is given by 
\[
\left(\int_{0}^{\infty}r^{n-1}\left(\int_{-\infty}^{\infty}\int_{S^{n-1}}\left|f\left(r,\theta,z\right)\right|^{2}d\theta dz\right)^{\frac{p}{2}}dr\right)^{\frac{1}{p}}.
\]
We can state our result.
\begin{thm}
Let $\Gamma$ be a compact surface of revolution, then the restriction
of the Fourier transform to $\Gamma$ is a bounded operator from $L_{rad}^{p}L_{zen}^{2}L_{ang}^{2}\left(\mathbb{R}^{n+1}\right)$
to $L^{2}\left(\Gamma\right)$, i.e. there eists a finite constant
$C_{p}$ such that
\begin{multline}
\left(\int_{-\infty}^{\infty}\int_{S^{n-1}}g\left(z\right)^{n-1}\sqrt{1+g'\left(z\right)^{2}}\left|\hat{f}\left(g\left(z\right),\theta,z\right)\right|^{2}d\theta dz\right)^{\frac{1}{2}}\lesssim\\
\lesssim C_{p}\left\Vert f\right\Vert _{L_{rad}^{p}L_{zen}^{2}L_{ang}^{2}\left(\mathbb{R}^{n+1}\right)},
\end{multline}
so long as $1\leq p<\frac{2n}{n+1}$.
\end{thm}
A central point in this area is C. Fefferman's observation that the
disc multiplier in $\mathbb{R}^{n}$ for $n\geq2$, given by the formula
\[
\widehat{T_{0}f}\left(\xi\right)=\chi_{B\left(0,1\right)}\left(\xi\right)\hat{f}\left(\xi\right),
\]
is bounded on $L^{p}\left(\mathbb{R}^{n}\right)$ only in the trivial
case $p=2$. However, it was later proved (see ref \cite{key-8} and
\cite{key-9}) that $T_{0}$ is bounded on the mixed norm spaces $L_{rad}^{p}L_{ang}^{2}\left(\mathbb{R}^{n}\right)$
if and only if $\frac{2n}{n+1}<p<\frac{2n}{n-1}$. Here we extend
that result to a more general class of radial multipliers. 
\begin{thm}
\label{thm:Ball}Let $T_{m}$ be a Fourier multiplier defined by 
\begin{equation}
\left(T_{m}f\right)\hat{}\left(\xi\right):=m\left(\left|\xi\right|\right)\hat{f}\left(\xi\right),\label{eq:Definition}
\end{equation}
for all rapidly decreasing smooth functions $f$, where $m$ satisfies
the following hypothesis:
\begin{enumerate}
\item $\text{\emph{Supp}}\left(m\right)\subset\left[a,b\right]\subset\mathbb{R}^{+}$,
and $m$ is differentiable in the interior $\left(a,b\right)$.
\item $\int_{a}^{b}\left|m'\left(x\right)\right|dx<\infty.$ 
\end{enumerate}
$T_{m}$ is then bounded in $L_{rad}^{p}L_{ang}^{2}\left(\mathbb{R}^{n}\right)$
so long as $\frac{2n}{n+1}<p<\frac{2n}{n-1}$.
\end{thm}
Finally, let us observe that Theorem \ref{thm:Ball} admits different
extensions taking into account Littlewood-Paley theory. Some vector
valued and weighted inequalities are satisfied by $T_{0}$ and the
so called \emph{universal Kakeya maximal function }acting on radial
functions.

\section{Restriction in the discrete setting}
\begin{proof}[Proof of Theorem \ref{thm:Discrete}]
 First let us observe that, by an easy argument, we can assume $M=1$
without loss of generality. Next we take a smooth cut-off $\varphi$
sot that 
\begin{eqnarray*}
\varphi & \equiv & 1\text{ on }B\left(0,\frac{1}{2}\right),\\
\varphi & \equiv & 0\text{ when }\left\Vert x\right\Vert \geq1,\\
\varphi & \in & C_{0}^{\infty}\left(\mathbb{R}^{2}\right).
\end{eqnarray*}
We can then write 
\begin{eqnarray*}
f\left(\xi\right) & = & \sum_{k}a_{k}\varphi\left(\xi+\xi_{k}\right)e^{2\pi i\xi\cdot q}\\
 & = & \sum_{k}a_{k}\varphi_{k}\left(\xi\right)e^{2\pi i\xi\cdot q},
\end{eqnarray*}
where $q$ is a point in $\mathbb{R}^{2}$. We have 
\[
\hat{f}\left(x\right)=\sum_{k}a_{k}\hat{\varphi}\left(x-q\right)e^{2\pi i\xi_{k}\cdot\left(x-q\right)}.
\]
Note that the $L^{4}$ norm of $\hat{f}$ majorizes the left hand
side of (\ref{eq:Discrete}),
\begin{eqnarray*}
\int\left|\hat{f}\left(x\right)\right|^{4}dx & \geq & \int_{x-q\in Q_{0}}\left|\sum a_{k}e^{2\pi i\xi_{k}\cdot\left(x-q\right)}\hat{\varphi}\left(x-q\right)\right|^{4}dx\\
 & \gtrsim & \int_{Q}\left|\sum a_{k}e^{2\pi i\xi_{k}\cdot x}\right|^{4}dx,
\end{eqnarray*}
where $Q_{0}=\left[-\frac{1}{2},\frac{1}{2}\right]^{2}$ and $Q=q+Q_{0}$.

On the other hand, we have 
\begin{eqnarray*}
\int\left|\hat{f}\left(x\right)\right|^{4}dx & = & \int\left|f\ast f\left(\xi\right)\right|^{2}d\xi\\
 & = & \int\left|\sum_{k,j}a_{k}a_{j}\varphi_{k}\ast\varphi_{j}\left(\xi\right)e^{i\xi\cdot q}\right|^{2}d\xi.
\end{eqnarray*}
Furthermore, because the supports of $\varphi_{k}$ and $\varphi_{j}$
have a finite overlapping, uniformly on the radius $R$. 
\[
\int\left|\hat{f}\left(x\right)\right|^{4}dx\lesssim\left(\sum\left|a_{k}\right|^{2}\right)^{2},
\]
q.e.d. 
\end{proof}
Using similar arguments we can obtain the following analogous result:
In $\mathbb{R}^{2}$ let us consider the parabola $\gamma\left(t\right)=\left(t,t^{2}\right)$
and a set of real numbers $\left\{ \xi_{j}\right\} $ so that $\left|t_{j+1}-t_{j}\right|\geq1$,
then
\[
\sup_{\mu\left(Q\right)=1}\left\Vert \sum_{j}a_{j}e^{2\pi i\gamma\left(t_{j}\right)\cdot x}\right\Vert _{L^{4}\left(Q\right)}\lesssim\left(\sum\left|a_{j}\right|^{2}\right)^{\frac{1}{2}}.
\]
An interesting open question is to decide if the $L^{4}$ norm could
be replaced by an $L^{p}$ norm ($p>4$) in the inequality above.
It is known that $p=6$ fails, but for $4<p<6$ it is, as far as we
know, an interesting open problem \cite{key-10}.

\section{The restriction conjecture in mixed norm spaces}

Recall that in $\mathbb{R}^{n+1}$ we establish cylindrical coordinates
$\left(r,\theta,z\right)$, where $\left(r,\theta\right)$ corresponds
to the usual spherical coordinates in $\mathbb{R}^{n}$ and $z\in\mathbb{R}$
denotes the zenithal component. We will also use the notation $\left(\rho,\phi,\zeta\right)$
to refer to the same coordinate system.

The $L_{rad}^{p}L_{zen}^{2}L_{ang}^{2}\left(\mathbb{R}^{n+1}\right)$
norm is therefore given by 
\begin{equation}
\left\Vert f\right\Vert _{L^{p,2,2}}=\left(\int_{0}^{\infty}r^{n-1}\left(\int_{-\infty}^{\infty}\int_{S^{n-1}}\left|f\left(r,\theta,z\right)\right|^{2}d\theta dz\right)^{\frac{p}{2}}dr\right)^{\frac{1}{p}}.
\end{equation}

Let $g$ be a continuous positive function supported on a compact
interval $I$ of the real line that is almost everywhere differentiable,
and consider the surface of revolution in $\mathbb{R}^{n+1}$ given
by

\begin{equation}
\Gamma:=\left\{ \left(g\left(z\right),\theta,z\right)\in\mathbb{R}^{n+1},\ \theta\in S^{n-1},-\infty<z<\infty\right\} .
\end{equation}
We are interested in the restriction to $\Gamma$ of the Fourier transform
of functions in the Schwartz class $\mathcal{S}\left(\mathbb{R}^{n+1}\right)$.
The restriction inequality 
\[
\left\Vert \hat{f}\right\Vert _{L^{2}\left(\Gamma\right)}\leq C_{p}\left\Vert f\right\Vert _{L^{p,2,2}\left(\mathbb{R}^{n+1}\right)}
\]
for $1\leq p<\frac{2n}{n+1}$ is, by duality, equivalent to the extension
estimate:
\[
\left\Vert \widehat{fd\Gamma}\right\Vert _{L^{q,2,2}\left(\mathbb{R}^{n+1}\right)}\leq C_{q}\left\Vert f\right\Vert _{L^{2}\left(\Gamma\right)}
\]
for $q>\frac{2n}{n-1}$.

To compute $\widehat{fd\Gamma}$ let us recall 
\begin{eqnarray*}
d\Gamma & = & g\left(z\right)^{n-1}\sqrt{1+\left(g'\left(z\right)\right)^{2}}dzd\theta\\
 & = & G_{1}\left(z\right)dzd\theta,
\end{eqnarray*}
so that 
\begin{equation}
\widehat{fd\Gamma}\left(\rho,\phi,\zeta\right)=\int_{-\infty}^{\infty}\int_{S^{n-1}}G_{1}\left(z\right)f\left(z,\theta\right)e^{-iz\zeta}e^{-i\left(\rho g\left(z\right)\right)\theta\cdot\phi}d\theta dz.\label{eq:Fourier1}
\end{equation}
Next we use the spherical harmonic expansion 
\[
f\left(z,\theta\right)=\sum_{k,j}a_{k,j}\left(z\right)Y_{k}^{j}\left(\theta\right),
\]
where for each $k$, $\left\{ Y_{k}^{j}\right\} _{j=1,\ldots,d\left(k\right)}$
is an orthonormal basis of the spherical harmonics degree $k$. We
then obtain:

\begin{multline*}
\widehat{fd\Gamma}\left(\rho,\phi,\zeta\right)=\sum_{k,j}2\pi i^{k}Y_{k}^{j}\left(\phi\right)\rho^{-\frac{n-2}{2}}\int_{-\infty}^{\infty}g\left(z\right)^{\frac{n}{2}}\left(1+\left(g'\left(z\right)\right)^{2}\right)^{\frac{1}{2}}\cdot\\
\cdot a_{k,j}\left(z\right)J_{k+\frac{n-2}{2}}\left(\rho g\left(z\right)\right)e^{-iz\zeta}dz,
\end{multline*}
where $J_{\nu}$ denotes Bessel's function of order $\nu$ (see ref.
\cite{key-13}). Denoting by $G_{2}\left(z\right):=g\left(z\right)^{\frac{n}{2}}\left(1+\left(g'\left(z\right)\right)^{2}\right)^{\frac{1}{2}}$,
the Fourier transform $\widehat{fd\Gamma}$ becomes 
\begin{equation}
\sum_{k,j}2\pi i^{k}Y_{k}^{j}\left(\phi\right)\rho^{-\frac{n-2}{2}}\int_{-\infty}^{\infty}G_{2}\left(z\right)a_{k,j}\left(z\right)J_{k+\frac{n-2}{2}}\left(\rho g\left(z\right)\right)e^{-iz\zeta}dz.
\end{equation}

Taking into account the orthogonality of the elements of the basis
$\left\{ Y_{k}^{j}\right\} $ together with Plancherel's Theorem in
the $z$-variable, we obtain that the mixed norm $\left\Vert \widehat{fd\Gamma}\right\Vert _{L^{q,2,2}}^{q}$
is up to a constant equal to 
\begin{gather}
\int_{0}^{\infty}\rho^{-q\frac{n-2}{2}+n-1}\left(\sum_{k,j}\int_{-\infty}^{\infty}\left|g\left(\zeta\right)\right|^{n}\left|1+\left(g'\left(\zeta\right)\right)^{2}\right|\left|a_{k,j}\left(\zeta\right)\right|^{2}\left|J_{\nu_{k}}\left(\rho g\left(\zeta\right)\right)\right|^{2}d\zeta\right)^{\frac{q}{2}}d\rho,
\end{gather}
where $\nu_{k}=k+\frac{n-2}{2}$. On the other hand we have 
\begin{eqnarray}
\int_{\Gamma}\left|f\right|^{2} & = & \int_{-\infty}^{\infty}\int_{S^{n-1}}\left|\sum_{j,k}a_{k,j}\left(z\right)Y_{k}^{j}\left(\theta\right)\right|^{2}g\left(z\right)^{n-1}\sqrt{1+g'\left(z\right)^{2}}d\theta dz\nonumber \\
 & = & \sum_{j,k}\int_{-\infty}^{\infty}\left|a_{k,j}\left(z\right)\right|^{2}g\left(z\right)^{n-1}\sqrt{1+g'\left(z\right)^{2}}dz.
\end{eqnarray}
Therefore our theorem will be a consecuence of the following fact:
\begin{lem}
\label{lem:Restriction}Given any sequence of positive indices $\left\{ \nu_{j}\right\} $
with $\nu_{j}\geq\frac{n-2}{2}$ for all $j$ and Schwartz functions
$a_{j}$, the following inequality holds:
\begin{multline}
\int_{0}^{\infty}\rho^{-q\frac{n-2}{2}+n-1}\left(\sum_{j}\int_{-\infty}^{\infty}\left|g\left(z\right)\right|^{n}\left|1+\left(g'\left(z\right)\right)^{2}\right|\left|a_{j}\left(z\right)\right|^{2}\left|J_{\nu_{j}}\left(\rho g\left(z\right)\right)\right|^{2}dz\right)^{\frac{q}{2}}d\rho\\
\lesssim\left(\sum_{j}\int_{-\infty}^{\infty}\left|g\left(z\right)\right|^{n-1}\left(1+\left(g'\left(z\right)\right)^{2}\right)^{\frac{1}{2}}\left|a_{j}\left(z\right)\right|^{2}dz\right)^{\frac{q}{2}},\label{eq:Restriction2}
\end{multline}
for $q>\frac{2n}{n-1}$. \end{lem}
\begin{rem}
Taking into account the hypothesis about $g$ we will look for estimates
depending upon $A=\sup_{x\in I}\left|g\left(x\right)\right|$ and
$B=\sup_{x\in I}\left|g'\left(x\right)\right|$, where $I$ is the
compact support of $g$. It is also easy to see that we can reduce
ourselves to consider the sums over the family of indices $\left\{ \nu_{j}\right\} _{j=1}^{\infty}$
such that $\nu_{j}\geq\frac{n-2}{2}$. Therefore it is enough to show
\begin{multline}
\int_{0}^{\infty}\rho^{-q\frac{n-2}{2}+n-1}\left(\sum_{j}\int_{-\infty}^{\infty}\left|b_{j}\left(z\right)\right|^{2}\left|J_{\nu_{j}}\left(\rho g\left(z\right)\right)\right|^{2}dz\right)^{\frac{q}{2}}d\rho\\
\lesssim\left(\sum_{j}\int_{-\infty}^{\infty}\left|b_{j}\left(z\right)\right|^{2}dz\right)^{\frac{q}{2}}\label{eq:Restriction3}
\end{multline}
for a family of smooth functions $\left\{ b_{j}\right\} _{j}$ and
indexes $\nu_{j}\geq\frac{n-2}{2}.$
\end{rem}
In order to show (\ref{eq:Restriction3}) we will need a sharp control
of the decay of Bessel functions; namely the following estimates:
\begin{lem}
\label{lem:Decay}The following estimates hold for $\nu\geq1$.
\begin{enumerate}
\item $J_{\nu}\left(r\right)\leq\frac{1}{r^{1/2}}$, when $r\geq2\nu$.
\item $J_{\nu}\left(r\right)\leq\frac{1}{\nu}$, when $r\leq\frac{1}{2}\nu$.
\item $J_{\nu}\left(\nu+\rho\nu^{1/3}\right)\leq\frac{1}{\rho^{1/4}\nu^{1/3}}$,
when $0\leq\rho\leq\frac{3}{2}\nu^{2/3}$.
\item $J_{\nu}\left(\nu-\rho\nu^{1/3}\right)\leq\frac{1}{\rho\nu^{1/3}}$,
when $1\leq\rho\leq\frac{3}{2}\nu^{2/3}$.
\item $J_{\nu}$$\left(r\right)\leq r^{\nu},$ as $r\to0$.
\end{enumerate}
\end{lem}
These assymptotics follow by the stationary phase method as it is
shown in \cite{key-13}, \cite{key-14} and \cite{key-15}.
\begin{proof}[Proof of Lemma \ref{lem:Restriction} ]
To prove \ref{eq:Restriction3} we shall first decompose the $\rho$-integration
in dyadic parts: $[0,\infty)=[0,1)\bigcup\cup_{n=0}^{\infty}[2^{n},2^{n+1})$. 

\begin{multline}
\int_{0}^{1}\rho^{-q\frac{n-2}{2}+n-1}\left(\sum_{j}\int_{-\infty}^{\infty}\left|b_{j}\left(z\right)\right|^{2}\left|J_{\nu_{j}}\left(\rho g\left(z\right)\right)\right|^{2}dz\right)^{\frac{q}{2}}d\rho\\
+\sum_{M}\int_{M}^{2M}\rho^{-q\frac{n-2}{2}+n-1}\left(\sum_{j}\int_{-\infty}^{\infty}\left|b_{j}\left(z\right)\right|^{2}\left|J_{\nu_{j}}\left(\rho g\left(z\right)\right)\right|^{2}dz\right)^{\frac{q}{2}}d\rho,\label{eq:restriction dyadic}
\end{multline}
where $M=2^{m},$ $m=0,1,\ldots$ 

For the lower integrand, we have the following splitting: 
\begin{eqnarray*}
\int_{0}^{1}\rho^{-q\frac{n-2}{2}+n-1}\left[\ldots\right]^{\frac{q}{2}}d\rho & = & \int_{0}^{\frac{1}{A}}\rho^{-q\frac{n-2}{2}+n-1}\left[\ldots\right]^{\frac{q}{2}}d\rho+\int_{\frac{1}{A}}^{1}\rho^{-q\frac{n-2}{2}+n-1}\left[\ldots\right]^{\frac{q}{2}}d\rho\\
 & = & I+II.
\end{eqnarray*}
In order to bound $I$ we invoke Minkowski's inequality and property
5. of Lemma \ref{lem:Decay}. 
\begin{eqnarray*}
I & \lesssim & \left[\int_{-\infty}^{\infty}\sum_{j}\left(\int_{0}^{\frac{1}{A}}\left\{ \rho^{-\left(n-2\right)+\frac{2}{q}\left(n-1\right)}\left|b_{j}\left(z\right)\right|^{2}\left|J_{\nu_{j}}\left(\rho z\right)\right|^{2}\right\} ^{\frac{q}{2}}d\rho\right)^{\frac{2}{q}}dz\right]^{\frac{q}{2}}.\\
 & \leq & \left[\int_{-\infty}^{\infty}\sum_{j}\left|b_{j}\left(z\right)\right|^{2}A^{2\nu_{j}}\left(\int_{0}^{\frac{1}{A}}\rho^{-q\frac{n-2}{2}+\left(n-1\right)+q\nu_{j}}d\rho\right)^{\frac{2}{q}}dz\right]^{\frac{q}{2}},
\end{eqnarray*}
where $A=\left\Vert g\right\Vert _{\infty}$. Since the sum is taken
over all $\nu_{j}\geq\frac{n-2}{2}$, the inner integrand is well
defined and we can bound 
\begin{equation}
I\lesssim A^{q\frac{n-1}{2}-n}\left[\sum_{j}\int_{-\infty}^{\infty}\left|b_{j}\left(z\right)\right|^{2}dz\right]^{\frac{q}{2}}.
\end{equation}
The second part is similarly bounded 
\begin{equation}
II\lesssim\left(1+A^{q\frac{n-1}{2}-n}\right)\left[\sum_{j}\int_{-\infty}^{\infty}\left|b_{j}\left(z\right)\right|^{2}dz\right]^{\frac{q}{2}}.
\end{equation}

Then Lemma \ref{lem:Restriction} will be a consequence of the following
claim:
\begin{claim}
\label{calim}For all $q>4$, the following inequality holds true

\begin{multline}
\int_{M}^{2M}\rho\left(\sum_{j}\int_{-\infty}^{\infty}\left|b_{j}\left(z\right)\right|^{2}\left|J_{\nu_{j}}\left(\rho g\left(z\right)\right)\right|^{2}dz\right)^{\frac{q}{2}}d\rho\\
\lesssim M^{\frac{4-q}{2}}\left(\int_{-\infty}^{\infty}\sum_{j}\left|b_{j}\left(z\right)\right|^{2}dz\right)^{\frac{q}{2}}.\label{eq:dimension2}
\end{multline}

\end{claim}
Indeed, if $q>4$ we need only to note that 
\begin{multline*}
\int_{M}^{2M}\rho^{-q\frac{n-2}{2}+n-1}\left(\sum_{j}\int_{-\infty}^{\infty}\left|b_{j}\left(z\right)\right|^{2}\left|J_{\nu_{j}}\left(\rho g\left(z\right)\right)\right|^{2}dz\right)^{\frac{q}{2}}d\rho\\
\lesssim M^{\left(n-2\right)\left(-\frac{q}{2}+1\right)}\int_{M}^{2M}\rho\left(\sum_{j}\int_{-\infty}^{\infty}\left|b_{j}\left(z\right)\right|^{2}\left|J_{\nu_{j}}\left(\rho g\left(z\right)\right)\right|^{2}dz\right)^{\frac{q}{2}}d\rho,
\end{multline*}
invoke our claim and sum over all dyadic intervals in (\ref{eq:restriction dyadic}):
\begin{multline}
\sum_{m}\int_{2^{m}}^{2^{m+1}}\rho^{-q\frac{n-2}{2}+n-1}\left(\sum_{j}\int_{-\infty}^{\infty}\left|b_{j}\left(z\right)\right|^{2}\left|J_{\nu_{j}}\left(\rho g\left(z\right)\right)\right|^{2}dz\right)^{\frac{q}{2}}d\rho\\
\lesssim\sum_{m}2^{m\left(n-2\right)\left(-\frac{q}{2}+1\right)+m\frac{4-q}{2}}\left(\int_{-\infty}^{\infty}\sum_{j}\left|b_{j}\left(z\right)\right|^{2}dz\right)^{\frac{q}{2}}.
\end{multline}
It is then a simple matter to check that the exponent is negative
for $q>\frac{2n}{n-1}$.

If the exponent $q$ is however smaller, $\frac{2n}{n-1}<q\leq4$,
we need to use an extra trick. Note that equation (\ref{eq:dimension2})
implies 
\[
\int_{M}^{2M}\left(\sum_{j}\int_{-\infty}^{\infty}\left|b_{j}\left(z\right)\right|^{2}\left|J_{\nu_{j}}\left(\rho g\left(z\right)\right)\right|^{2}dz\right)^{\frac{q_{1}}{2}}d\rho\lesssim M^{1-\frac{q_{1}}{2}}\left(\int_{-\infty}^{\infty}\sum_{j}\left|b_{j}\left(z\right)\right|^{2}dz\right)^{\frac{q_{1}}{2}},
\]
for all $q_{1}>4$. Then using Hölder's inequality and the previous
inequality, 
\begin{multline*}
\int_{M}^{2M}\left(\sum_{j}\int_{-\infty}^{\infty}\left|b_{j}\left(z\right)\right|^{2}\left|J_{\nu_{j}}\left(\rho g\left(z\right)\right)\right|^{2}dz\right)^{\frac{q}{2}}d\rho\\
\lesssim M^{1-\frac{q}{q_{1}}}\left(\int_{M}^{2M}\left(\sum_{j}\int_{-\infty}^{\infty}\left|b_{j}\left(z\right)\right|^{2}\left|J_{\nu_{j}}\left(\rho g\left(z\right)\right)\right|^{2}dz\right)^{\frac{q_{1}}{2}}d\rho\right)^{\frac{q}{q_{1}}}.
\end{multline*}
Therefore, summing over all intervals, we obtain 
\begin{multline*}
\sum_{m}\int_{2^{m}}^{2^{m+1}}\rho^{-q\frac{n-2}{2}+n-1}\left(\sum_{j}\int_{-\infty}^{\infty}\left|b_{j}\left(z\right)\right|^{2}\left|J_{\nu_{j}}\left(\rho g\left(z\right)\right)\right|^{2}dz\right)^{\frac{q}{2}}d\rho\\
\lesssim\sum_{m}2^{m\left\{ -q\frac{n-2}{2}+n-1+1-\frac{q}{2}\right\} }\left(\int_{-\infty}^{\infty}\sum_{j}\left|b_{j}\left(z\right)\right|^{2}dz\right)^{\frac{q}{2}},
\end{multline*}
where the exponent $-q\frac{n-1}{2}+n$ is negative for all $q>\frac{2n}{n-1}$
.

To prove Claim \ref{calim} let us split each dyadic integrand in
(\ref{eq:restriction dyadic}) in three parts corresponding to the
differnt ranges of control of Bessel functions.
\begin{align*}
 & \int_{M}^{2M}\rho\left(\int_{-\infty}^{\infty}\sum_{\nu_{j}\in I^{0}}\left|b_{j}\left(z\right)\right|^{2}\left|J_{\nu_{j}}\left(\rho g\left(z\right)\right)\right|^{2}dz\right)^{\frac{q}{2}}d\rho\\
 & +\int_{M}^{2M}\rho\left(\int_{-\infty}^{\infty}\sum_{\nu_{j}\in I^{c}}\left|b_{j}\left(z\right)\right|^{2}\left|J_{\nu_{j}}\left(\rho g\left(z\right)\right)\right|^{2}dz\right)^{\frac{q}{2}}d\rho\\
 & +\int_{M}^{2M}\rho\left(\int_{-\infty}^{\infty}\sum_{\nu_{j}\in I^{\infty}}\left|b_{j}\left(z\right)\right|^{2}\left|J_{\nu_{j}}\left(\rho g\left(z\right)\right)\right|^{2}dz\right)^{\frac{q}{2}}d\rho\\
 & =\sum_{M}\left(I_{M}^{0}+I_{M}^{c}+I_{M}^{\infty}\right),
\end{align*}
where $I_{M}^{0}=\left[0,Mg\left(z\right)/2\right)$, $I_{M}^{c}=\left[Mg\left(z\right)/2,4Mg\left(z\right)\right)$,
and $I_{M}^{\infty}=\left[4Mg\left(z\right),\infty\right)$.

Recall that if $2k<r$, $\left|J_{k}\left(r\right)\right|\leq r^{-1/2}$;
in $I_{M}^{0}$ we have $2\nu_{j}<Mg\left(z\right)<\rho g\left(z\right)$,
hence
\begin{eqnarray}
I_{M}^{0} & \leq A^{-\frac{q}{2}} & \int_{M}^{2M}\rho^{1-\frac{q}{2}}\left(\int_{-\infty}^{\infty}\sum_{\nu_{j}\in I^{0}}\left|b_{j}\left(z\right)\right|^{2}dz\right)^{\frac{q}{2}}d\rho\nonumber \\
 & \leq A^{-\frac{q}{2}} & M^{\frac{4-q}{2}}\left(\int_{-\infty}^{\infty}\sum_{\nu_{j}}\left|b_{j}\left(z\right)\right|^{2}dz\right)^{\frac{q}{2}}.
\end{eqnarray}

Similarly, $I_{M}^{\infty}$ is also easily bounded as if $k>2r$,
$\left|J_{k}\left(r\right)\right|\leq k^{-1}$, and in $I_{M}^{\infty}$,
$k>4Mg\left(z\right)>2\rho g\left(z\right)$. Furthermore, since $\rho g\left(z\right)>1$,
$\left(\rho g\left(z\right)\right)^{-2}<\left(\rho g\left(z\right)\right)^{-1}$
and, in $I_{M}^{\infty}$, we have $\left|J_{k}\left(\rho g\left(z\right)\right)\right|^{2}\leq\left(\rho g\left(z\right)\right)^{-1}$.
This shows that again 
\begin{equation}
I_{M}^{\infty}\leq A^{-\frac{q}{2}}M^{\frac{4-q}{2}}\left(\int_{-\infty}^{\infty}\sum_{\nu_{j}}\left|b_{j}\left(z\right)\right|^{2}dz\right)^{\frac{q}{2}}.
\end{equation}

Finally, we need to work a little bit harder than in the previous
cases to obtain a suitable estimate for $I_{M}^{c}$. First of all
note that Minkowski's inequality yields
\begin{equation}
I_{M}^{c}\leq\left[\int_{-\infty}^{\infty}\left\{ \int_{M}^{2M}\rho\left(\sum_{\nu_{j}\in I^{c}}\left|b_{j}\left(z\right)\right|^{2}\left|J_{\nu_{j}}\left(\rho g\left(z\right)\right)\right|^{2}\right)^{\frac{q}{2}}d\rho\right\} ^{\frac{2}{q}}dz\right]^{\frac{q}{2}}.\label{eq:Minkowski}
\end{equation}
In $I_{M}^{c}$ we want to use estimate (3) of Lemma \ref{lem:Decay},
we thus need to split the inner integral so that $\rho g\left(z\right)\sim\nu_{j}+\alpha\nu_{j}$
in the according range of $\alpha$. Consider the family of sets 
\[
G_{\alpha}=\left[\frac{M}{2}+\alpha M^{\frac{1}{3}}g\left(z\right)^{-\frac{2}{3}},\frac{M}{2}+\left(\alpha+1\right)M^{\frac{1}{3}}g\left(z\right)^{-\frac{2}{3}}\right],
\]
for $\alpha=0,1,2,\ldots,\left[\left(Mg\left(z\right)\right)^{\frac{2}{3}}\right]$,
so that $\bigcup G_{\alpha}\supseteq\left[M,2M\right]$ and in each
interval $\rho g\left(z\right)\sim\nu_{j}+\alpha\nu_{j}^{\frac{1}{3}}$,
and split (\ref{eq:Minkowski}) in the following way 
\[
I_{M}^{c}\lesssim\left[\int_{-\infty}^{\infty}\left\{ \sum_{\alpha}\int_{G_{\alpha}}\rho\left(\sum_{\nu_{j}\in I^{c}}\left|b_{j}\left(z\right)\right|^{2}\left|J_{\nu_{j}}\left(\rho g\left(z\right)\right)\right|^{2}\right)^{\frac{q}{2}}d\rho\right\} ^{\frac{2}{q}}dz\right]^{\frac{q}{2}},
\]

Let us also define 
\[
A_{\beta}=\sum_{\nu_{j}\in G_{\beta}}\left|b_{j}\left(z\right)\right|^{2}.
\]
We can then invoke Lemma \ref{lem:Decay} and rearragne the sums to
bound $I_{M}^{c}$ by 
\begin{multline*}
\left[\int_{-\infty}^{\infty}\left\{ \sum_{\alpha}\int_{G_{\alpha}}\left(\sum_{\beta\leq\alpha}A_{\beta}\frac{1}{\left(\left|\alpha-\beta\right|+1\right)^{1/2}M^{\frac{2}{3}}g\left(z\right)^{-\frac{4}{3}}}\right)^{\frac{q}{2}}\rho d\rho\right\} ^{\frac{2}{q}}dz\right]^{\frac{q}{2}}\\
+\left[\int_{-\infty}^{\infty}\left\{ \sum_{\alpha}\int_{G_{\alpha}}\left(\sum_{\beta\geq\alpha}A_{\beta}\frac{1}{\left(\left|\alpha-\beta\right|+1\right)^{2}M^{\frac{2}{3}}g\left(z\right)^{-\frac{4}{3}}}\right)^{\frac{q}{2}}\rho d\rho\right\} ^{\frac{2}{q}}dz\right]^{\frac{q}{2}}.
\end{multline*}
Note that the second sum is easier to control than the first. We shall,
therefore, focus on the first term, $I_{M,}^{c,1}$. Since the intervals
$G_{\alpha}$ have length $M^{\frac{1}{3}}g\left(z\right)^{-\frac{2}{3}}$,
\begin{eqnarray*}
I_{M}^{c,1} & \lesssim & M^{\frac{4-q}{3}}A^{2\frac{\left(q-1\right)}{3}}\left[\int_{-\infty}^{\infty}\left\{ \sum_{\alpha}\left(\sum_{\beta\geq\alpha}A_{\beta}\frac{1}{\left(\left|\alpha-\beta\right|+1\right)^{2}}\right)^{\frac{q}{2}}\right\} ^{\frac{2}{q}}dz\right]^{\frac{q}{2}}.
\end{eqnarray*}
Furthermore, using Young's inequality, since $q>4$, taking $2/q=1/s-1/2$
we obtain 
\begin{eqnarray*}
\sum_{\alpha}\left(\sum_{\beta\geq\alpha}A_{\beta}\frac{1}{\left(\left|\alpha-\beta\right|+1\right)^{2}}\right)^{\frac{q}{2}} & \lesssim & \left(\sum_{\gamma}A_{\gamma}^{s}\right)^{\frac{q}{2s}}\\
 & \lesssim & \left(\sum_{\gamma}A_{\gamma}\right)^{\frac{q}{2}}.
\end{eqnarray*}
We have thus ahowed that the central integrand $I_{M}^{c}$ can also
be bounded in the desired way; 
\begin{equation}
I_{M}^{c,1}\lesssim A^{2\left(\frac{q-1}{3}\right)}M^{\frac{4-q}{3}}\left[\int_{-\infty}^{\infty}\sum_{k\in I_{M}^{c}}\left|a_{k}\right|^{2}dz\right]^{\frac{q}{2}}.
\end{equation}
q.e.d.
\end{proof}

\section{Generalized Disc Multiplier}

In the late 80's it was proved independently in \cite{key-8,key-9}
that the disc multiplier operator is bounded in the mixed norm spaces
$L_{rad}^{p}L_{ang}^{2}\left(\mathbb{R}^{n}\right)$ for all $\frac{2n}{n+1}<p<\frac{2n}{n-1}$.
Let us here explore further the theory of radial fourier multipliers
following the ideas presented in the aforementioned articles. 

Let $m$ be a radial function and consider the fourier multiplier
\[
\left(T_{m}f\right)\hat{}\left(\xi\right)=m\left(\left|\xi\right|\right)\hat{f}\left(\xi\right).
\]
Once again, recall the expansion of a given function $f$ in terms
of its spherical harmonics, 
\[
f\left(x\right)=\sum_{k=0}^{\infty}\sum_{j=1}^{d\left(k\right)}f_{k,j}\left(\left|x\right|\right)Y_{k}^{j}\left(\frac{x}{\left|x\right|}\right).
\]
Then, the classical formula relating the Fourier transform and the
spherical harmonics expansion, \cite{key-11-1}, yields 
\[
\hat{f}\left(\xi\right)=\sum_{k=0}^{\infty}\sum_{j=1}^{d\left(k\right)}Y_{k}^{j}\left(\frac{\xi}{\left|\xi\right|}\right)2\pi i^{k}\left|\xi\right|^{-\left(k+\frac{n-2}{2}\right)}\int_{0}^{\infty}f_{k,j}\left(t\right)J_{k+\frac{n-2}{2}}\left(2\pi\left|\xi\right|t\right)t^{k+\frac{n}{2}}dt.
\]
The expression of $T_{m}$ in terms of its spherical harmonics expansion
is then 
\begin{multline*}
T_{m}f\left(x\right)=\sum_{k=0}^{\infty}\sum_{j=1}^{d\left(k\right)}2\pi i^{k}\int_{\mathbb{R}^{n}}e^{2\pi ix\xi}m\left(\left|\xi\right|\right)Y_{k}^{j}\left(\frac{\xi}{\left|\xi\right|}\right)\left|\xi\right|^{-\left(k+\frac{n-2}{2}\right)}\\
\int_{0}^{\infty}f_{k,j}\left(t\right)J_{k+\frac{n-2}{2}}\left(2\pi\left|\xi\right|t\right)t^{k+\frac{n-2}{2}}dtd\xi.
\end{multline*}
Exchanging the order of integration, the previous expression becomes
\[
\sum_{k=0}^{\infty}\sum_{j=1}^{d\left(k\right)}2\pi i^{k}\int_{0}^{\infty}f_{k,j}\left(t\right)t^{k+\frac{n-2}{2}}\hat{g_{t}}\left(x\right)dx,
\]
where 
\[
g_{t}\left(\xi\right)=m\left(\left|\xi\right|\right)J_{k+\frac{n-2}{2}}\left(2\pi\left|\xi\right|t\right)\left|\xi\right|^{-\left(k+\frac{n-2}{2}\right)}Y_{k}^{j}\left(\frac{\xi}{\left|\xi\right|}\right).
\]
Therefore, computing once more the Fourier transform of a radial function,
\[
T_{m}f\left(r\theta\right)=\sum_{k=0}^{\infty}\sum_{j=1}^{d\left(k\right)}4\pi^{2}\left(-1\right)^{k}Y_{k}^{j}\left(\theta\right)T_{m}^{k,j}f\left(r\right),
\]
with 
\[
T_{m}^{k,j}f\left(r\right)=\int_{0}^{\infty}f_{k,j}\left(t\right)t^{\frac{n+2k-1}{2}}r^{-\frac{n+2k-1}{2}}K_{k+\frac{n-2}{2}}\left(t,r\right)dt,
\]
where 
\[
K_{\nu}\left(t,r\right)=\sqrt{rt}\int_{a}^{b}m\left(s\right)J_{\nu}\left(2\pi ts\right)J_{\nu}\left(2\pi rs\right)sds.
\]
In order to simplify the notation, note that
\begin{equation}
T_{m}f\left(r\theta\right)\approx\sum_{k=0}^{\infty}\sum_{j=1}^{d\left(k\right)}Y_{k}^{j}\left(\theta\right)T_{m}^{k,j}f\left(r\right)
\end{equation}
with $T_{m}^{k,j}$ defined as before, but 
\[
K_{\nu}\left(t,r\right)=\sqrt{rt}\int_{a}^{b}m\left(s\right)J_{\nu}\left(ts\right)J_{\nu}\left(rs\right)sds.
\]

Let us take a closer look at the kernel of the operator $K_{\alpha}$,
\begin{equation}
K_{\alpha}\left(t,r\right)=\sqrt{rt}\int_{a}^{b}m\left(s\right)J_{\alpha}\left(ts\right)J_{\alpha}\left(rs\right)sds.\label{eq:kernel}
\end{equation}
It is suitable to decode these kernels in terms of an auxiliary function
$\mathcal{U}_{r}\left(s\right)=\sqrt{rs}J_{\text{\ensuremath{\alpha}}}\left(rs\right)$.
The use of Bessel's equation yields 
\[
\frac{\partial}{\partial s}\left\{ \mathcal{U}_{r}\left(s\right)\mathcal{U}_{t}'\left(s\right)-\mathcal{U}_{t}\left(s\right)\mathcal{U}_{r}'\left(s\right)\right\} =\left(t^{2}-r^{2}\right)\sqrt{tr}J_{\alpha}\left(rs\right)J_{\alpha}\left(ts\right)s.
\]
Therefore, after an integration by parts in (\ref{eq:kernel}), we
obtain 
\begin{eqnarray*}
K_{\alpha}\left(t,r\right) & = & \left[m\left(s\right)\frac{1}{t^{2}-r^{2}}\left\{ \mathcal{U}_{r}\left(s\right)\mathcal{U}_{t}'\left(s\right)-\mathcal{U}_{t}\left(s\right)\mathcal{U}_{r}'\left(s\right)\right\} \right]_{a}^{b}\\
 &  & -\int_{a}^{b}m'\left(s\right)\frac{1}{t^{2}-r^{2}}\left\{ \mathcal{U}_{r}\left(s\right)\mathcal{U}_{t}'\left(s\right)-\mathcal{U}_{t}\left(s\right)\mathcal{U}_{r}'\left(s\right)\right\} ds.
\end{eqnarray*}
Hence, we express the modified disc multiplier in the following way
\begin{multline}
T_{m}f\left(r\theta\right)=\sum_{k,j}Y_{k}^{j}\left(\theta\right)\int_{0}^{\infty}f_{k,j}\left(t\right)t^{\frac{n+2k-1}{2}}r^{-\frac{n+2k-1}{2}}\cdot\\
\cdot\left(m\left(b\right)k\left(r,t,b\right)-m\left(a\right)k\left(r,t,a\right)-\int_{a}^{b}m'\left(s\right)k\left(r,t,s\right)ds\right)dt,\label{eq:decomposition}
\end{multline}
where $k_{\alpha}\left(t,r,s\right)$ denotes the kernel $\frac{1}{t^{2}-r^{2}}\left\{ \mathcal{U}_{r}\left(s\right)\mathcal{U}_{t}'\left(s\right)-\mathcal{U}_{t}\left(s\right)\mathcal{U}_{r}'\left(s\right)\right\} $.
A simple expansion of $k_{\alpha}$ reveals the underlying singularities
of the operator $K_{\alpha}$; 
\begin{eqnarray}
k_{\alpha}\left(t,r,s\right) & = & \bigg(s\frac{\sqrt{t}J_{\text{\ensuremath{\alpha}}}'\left(ts\right)J_{\alpha}\left(rs\right)\sqrt{r}}{2\left(t-r\right)}+s\frac{\sqrt{t}J_{\text{\ensuremath{\alpha}}}'\left(ts\right)J_{\alpha}\left(rs\right)\sqrt{r}}{2\left(t+r\right)}\nonumber \\
 &  & +s\frac{\sqrt{t}J_{\text{\ensuremath{\alpha}}}\left(ts\right)J_{\alpha}'\left(rs\right)\sqrt{r}}{2\left(r-t\right)}+s\frac{\sqrt{t}J_{\text{\ensuremath{\alpha}}}'\left(ts\right)J_{\alpha}\left(rs\right)\sqrt{r}}{2\left(t+r\right)}\Bigg).\label{eq:core}
\end{eqnarray}
A thorough study of the kernel $k_{\alpha}\left(r,t,1\right)$ was
carried out in \cite{key-8} using the decay properties of Bessel
functions (Lemma (\ref{lem:Decay})) in order to show that the disc
multiplier is bounded in the mixed norm spaces $L_{rad}^{p}L_{ang}^{2}\left(\mathbb{R}^{n}\right)$
in the optimal range $\frac{2n}{n+1}<p<\frac{2n}{n-1}$. 

Although nothing really new has been done, we have brought to light
a more general family of operators underlying the disc multiplier,
that is the family of operators $T^{s}$ defined as 
\begin{equation}
T^{s}f\left(r\theta\right)=\sum_{k,j}Y_{k}^{j}\left(\theta\right)\int_{0}^{\infty}f_{k,j}\left(t\right)t^{\frac{n+2k-1}{2}}r^{-\frac{n+2k-1}{2}}k\left(r,t,s\right)dt.\label{eq:generalized}
\end{equation}
Indeed, any bound on operators $T^{s}$ that is uniform in $s$ implies
a bound on $T_{m}$ for a suitable $m$.
\begin{prop}
Let $f$ be a rapidly decreasing function then, for every $\frac{2n}{n+1}<p<\frac{2n}{n-1}$
\begin{equation}
\left\Vert T^{s}f\right\Vert _{p,2}\leq C_{p,n}\left\Vert f\right\Vert _{p,2},\label{eq:bound}
\end{equation}
where the constant $C_{p,n}$ is uniform in $s$. \end{prop}
\begin{proof}
In order to simplify the expressions we will just write one of the
four core kernels of $k_{\alpha}$ apparent in (\ref{eq:core}), that
is 

\begin{equation}
T^{s}f\left(r\theta\right)\sim\sum_{k,j}Y_{k}^{j}\left(\theta\right)\int_{0}^{\infty}f_{k,j}\left(t\right)t^{\frac{n+2k-1}{2}}r^{-\frac{n+2k-1}{2}}s\frac{\sqrt{t}J_{\text{\ensuremath{\alpha}}}'\left(ts\right)J_{\alpha}\left(rs\right)\sqrt{r}}{2\left(t-r\right)}dt,
\end{equation}
for any fixed $s\in\mathbb{R}$. The orthonormality in $L^{2}\left(\mathbb{S}^{n-1}\right)$
of spherical harmonics can now be used in our advantage to complute
the $L_{rad}^{p}L_{ang}^{2}$ norm of $T^{s}$. Indeed, $\left\Vert T^{s}f\right\Vert _{p,2,}$is
up to the notation reduction equal to 
\begin{multline*}
\left(\int_{0}^{\infty}r^{n-1}\left\{ \sum_{k,j}\left|\int_{0}^{\infty}f_{k,j}\left(t\right)t^{\frac{n+2k-1}{2}}r^{-\frac{n+2k-1}{2}}s\frac{\sqrt{t}J_{\text{\ensuremath{\alpha}}}'\left(ts\right)J_{\alpha}\left(rs\right)\sqrt{r}}{2\left(t-r\right)}dt\right|^{2}\right\} ^{\frac{p}{2}}dr\right)^{\frac{1}{p}}.
\end{multline*}
Two simple changes of variables, $t'=st$ and $r'=sr$, yield
\[
s^{-\frac{n}{p}}\left(\int_{0}^{\infty}r^{n-1}\left\{ \sum_{k,j}\left|\int_{0}^{\infty}f_{k,j}\left(\frac{t}{s}\right)t^{\frac{n+2k-1}{2}}r^{-\frac{n+2k-1}{2}}\frac{\sqrt{t}J_{\text{\ensuremath{\alpha}}}'\left(t\right)J_{\alpha}\left(r\right)\sqrt{r}}{2\left(t-r\right)}dt\right|^{2}\right\} ^{\frac{p}{2}}dr\right)^{\frac{1}{p}}.
\]
Note that this expression corresponds to that of the disc multiplier
$T_{0}$ analyzed by in \cite{key-8}. We can therefore bound it by
\[
C_{p,n}s^{-\frac{n}{p}}\left(\int_{0}^{\infty}r^{n-1}\left\{ \sum_{k,j}\left|f_{k,j}\left(\frac{r}{s}\right)\right|^{2}\right\} ^{\frac{p}{2}}dr\right)^{\frac{1}{p}},
\]
for every $\frac{2n}{n+1}<p<\frac{2n}{n-1}$. One last change of variables
produces the estimate 
\[
\left\Vert T^{s}f\right\Vert _{p,2}\leq C\left\Vert f\right\Vert _{p,2},
\]
where $C$ is uniform on $s$.
\end{proof}
It is now a simple matter to produce a bound for the operator $T_{m}$.
\begin{equation}
\left\Vert T_{m}f\right\Vert _{p,2}\lesssim\left|m\left(b\right)\right|\left\Vert T^{b}f\right\Vert _{p,2}+\left|m\left(a\right)\right|\left\Vert T^{b}f\right\Vert _{p,2}+\int_{a}^{b}\left|m'\left(s\right)\right|\left\Vert T^{s}f\right\Vert _{p,2}ds,\label{eq:Subordination}
\end{equation}
and Theorem \ref{thm:Ball} follows from the uniformity in the bound
(\ref{eq:bound}). That is 
\[
\left\Vert T_{m}f\right\Vert _{p,2}\leq C\left(\sup_{s\in\left[a,b\right]}\left|m\left(s\right)\right|+\int_{a}^{b}\left|m'\left(s\right)\right|ds\right)\left\Vert f\right\Vert _{p,2}.
\]

\end{document}